\documentclass[10pt]{amsart}
\usepackage{amssymb, amscd, amsmath, amsthm, amscd, 
epsfig, latexsym, enumerate, graphics}

\renewcommand{\geq}{\geqslant}
\renewcommand{\leq}{\leqslant}

\newtheorem{theorem}{Theorem}
\newtheorem{lemma}[theorem]{Lemma}
\newtheorem{cor}[theorem]{Corollary}

\newtheorem*{lem}{Lemma}

\begin{document}
\title{Nilpotent groups with balanced presentations. II}
 
\author{J. A. Hillman}
\address{School of Mathematics and Statistics\\
     University of Sydney, NSW 2006\\
      Australia }

\email{jonathanhillman47@gmail.com}

\begin{abstract}
If $G$ is a nilpotent group with a balanced presentation and
$G\not\cong\mathbb{Z}^3$ then $\beta_1(G;\mathbb{Q})\leq2$
\cite{Hi22}.
We show that if such a group $G$ has an abelian normal subgroup $A$ 
such that $G/A\cong\mathbb{Z}^2$ then $G$ is torsion-free 
and has Hirsch length $h(G)\leq4$.
On the other hand, if $\beta_1(G;\mathbb{Q})=1$ and
$G$ has an abelian normal subgroup $A$ such that 
$G/A\cong\mathbb{Z}$ then 
$G\cong\mathbb{Z}/m\mathbb{Z}\rtimes_n\mathbb{Z}$, 
for some $m,n\not=0$ such that $m$ divides a power of $n-1$.
\end{abstract}



\maketitle
A finite presentation for a group $G$ is {\it balanced\/}
if it has an equal number of generators and relations.
This notion has been mostly studied in connection with finite groups, 
but here we shall consider finitely generated infinite nilpotent groups.
In \cite{Hi22} we showed that if such a group $G$ has a balanced presentation
then either $G\cong\mathbb{Z}^3$ or 
$\beta_1(G;\mathbb{Q})\leq2$, 
and if $G$ is also metabelian then it has Hirsch length $h(G)\leq4$.
There are infinite nilpotent groups with balanced presentations
and which have non-trivial torsion,
but all known examples have Hirsch length $h=1$.
We expect that in this case the torsion subgroup should be homologically balanced,
while nilpotent groups with balanced presentations
and Hirsch length $h>1$ should be torsion-free. 
Our methods confirm this for $G$ metabelian, nilpotent
and with torsion-free abelianization.

If a group $G$ has a balanced presentation then 
$\beta_2(G;F)\leq\beta_1(G;F)$ for all fields $F$.
Although this is not a sufficient condition for $G$ to have 
a balanced presentation,
it is the most useful necessary condition
and underlies all of our arguments.
In \cite{Hi22} the emphasis was on the rational field $F=\mathbb{Q}$, 
and the torsion subgroup could be ignored.
Here we are interested in delineating the possible torsion,
and so other coefficient fields are needed.
Our strategy shall be to show that if a finitely generated nilpotent group $G$ 
has a normal subgroup $K$ with quotient $\mathbb{Z}^r$, for $r=1$ or 2,
and $\beta_2(G;\mathbb{F}_p)\leq\beta_1(G;\mathbb{F}_p)$
then $H^2(K;\mathbb{F}_p)$ has no subgroup which is a direct sum 
of non-trivial subgroups invariant under the action of $Aut(K)$.
This test seems difficult to apply, except when $K$ is abelian, 
and so our main results are restricted to groups which 
are abelian extensions of $\mathbb{Z}^r$.

The first  section presents our notation and some basic facts,
and we summarize briefly what is currently known 
about finite nilpotent groups with balanced presentations.
The next three sections prepare for working with infinite nilpotent groups.
In \S5 we show  that if $G$ has a balanced presentation and 
$h(G)=1$ then $G$ is 2-generated,
and if the torsion subgroup of $G$ is abelian then 
$G\cong\mathbb{Z}/m\mathbb{Z}\rtimes_n\mathbb{Z}$, 
for some $m,n\geq1$ (Theorem \ref{h=1}).
In the final section we show that if $G$ has an abelian normal subgroup $A$ 
such that $G/A\cong\mathbb{Z}^2$ then $G$ is torsion-free 
and $A$ has rank $\leq2$ (Corollary \ref{Z2cor}).

\section{notation and generalities}

If $G$ is a group then $G'$, $G^{ab}=G/G'$,  $\zeta{G}$ and $|G|$
shall denote the commutator subgroup,  abelianization, centre 
and order of $G$, respectively.
The Hirsch length $h(S)$ of a solvable group $S$ is the sum 
of the ranks of the abelian sections of a composition series for $S$.
If $N$ is a finitely generated nilpotent group then it
is finitely presentable and $h(N)$ is finite.

A group is {\it $d$-generated} if it can be generated by $d$ elements
and has a {\it balanced presentation\/} if it has a finite
presentation with equal numbers of generators and relations.
We shall say that $G$ is {\it homologically balanced\/}
if $G$ is finitely generated and $\beta_2(G;F)\leq\beta_1(G;F)$, 
for all fields $F$.
If $G$ is a homologically balanced nilpotent group 
then $G$ is $3$-generated \cite{Lu83}.

If $G$ has a balanced presentation then it is homologically balanced, 
and if $G$ is also finite then it must have trivial multiplicator: 
$H_2(G;\mathbb{Z})=0$.
(These assertions follow most easily from consideration of the
homology of the 2-complex associated to a balanced presentation
for the group.)

If $G$ is finitely presentable then $H_i(G;R)$
is finitely generated for $i\leq2$ and all simple coefficients $R$.
It follows easily that $\beta_i(G;\mathbb{Q})=\beta_i(G;\mathbb{F}_p)$,
for $i\leq2$ and almost all primes $p$. 
Hence $\beta_2(G;\mathbb{Q})=\beta_1(G;\mathbb{Q})$
if and only if $\beta_2(G;\mathbb{F}_p)=\beta_1(G;\mathbb{F}_p)$,
for almost all primes $p$. 

The Universal Coefficient Theorem for homology gives an exact sequence
\[
0\to{F}\otimes{H_2(G;\mathbb{Z})}\to{H_2(G;F)}\to{Tor(F,G^{ab})}\to0,
\]
for any group $G$ and field $F$, since $H_1(G;\mathbb{Z})=G^{ab}$.
If $A$ is a finitely generated abelian group and $F=\mathbb{F}_p$ then 
$Tor(\mathbb{F}_p,A)\cong{_pA}=\mathrm{Ker}(p.id_A)$.
If $A$ is finite then $A/pA$ and $\mathrm{Ker}(p.id_A)$
have the same dimension.
Hence if $G$ is finite then
$\beta_2(G;\mathbb{F}_p)\geq\beta_1(G;\mathbb{F}_p)$,
for any prime $p$, and $G$ is homologically balanced if and only if 
$H_2(G;\mathbb{Z})=0$.

If $G$ is abelian then $G=A$ and $H_2(G;\mathbb{Z})=A\wedge{A}$.
If also $F=\mathbb{F}_p$,  and $p$ is odd or if $p=2$ and 
$A$ has no summand of exponent 2 then this sequence is canonically split, 
and so 
$H_2(A;\mathbb{F}_p)\cong((A/pA)\wedge(A/pA))\oplus\mathrm{Ker}(p.id_A)$
\cite[Theorem V.6.6]{Br}.
There are abelian 2-groups for which there is no canonical splitting \cite{IZ18}.

There are similar Universal Coefficent 
exact sequences for cohomology
\[
0\to{Ext(G^{ab},F)}\to{H^2(G;F)}\to{Hom(G^{ab},F)}\to0.
\]

A nilpotent group $G$ is finite if and only if $\beta_1(G;\mathbb{Q})=0$
if and only if $h(G)=0$.
The Sylow subgroups of a finite nilpotent group $G$ are characteristic,
and $G$ is the direct product of its Sylow subgroups \cite[5.2.4]{Ro}.
It then follows from the K\"unneth Theorem that $H_2(G;\mathbb{Z})=0$ 
if and only if $H_2(P;\mathbb{Z})=0$ for all such Sylow subgroups $P$.
On the other hand, it is not clear that if $H_2(G;\mathbb{Z})=0$ then
$N$ must have a balanced presentation, 
even if this is so for each of its Sylow subgroups.
In general there may be a gap between homological 
necessary conditions and combinatorial sufficient conditions.
(The examples in \cite{HW85} of finite groups with trivial multiplicator
but without balanced presentations are not nilpotent.)

A finite abelian group $A$ is homologically balanced if and only if it is cyclic,
for if $A/pA$ is not cyclic for some prime $p$ then 
$(A/pA)\wedge(A/pA)\not=0$, and so
$\beta_2(A;\mathbb{F}_p)>\beta_1(A;\mathbb{F}_p)$,
by the Universal Coefficient exact sequences of \S1.
Finite cyclic groups clearly have balanced presentations.

If $p$ is an odd prime then every 2-generator metacyclic $p$-group 
$P$ with $H_2(P;\mathbb{Z})=0$ has a balanced presentation
\[
\langle{a,b}\mid{b^{p^{r+s+t}}=a^{p^{r+s}}},~bab^{-1}=a^{1+p^r}\rangle,
\]
where $r\geq1$ and $s,t\geq0$.
(The order of such a group is $p^{3r+2s+t}$.)
There are other metacyclic 2-groups and other $p$-groups 
with 2-generator balanced presentations.
A handful of 3-generated $p$-groups (for $p=2$ and 3)
are also known to have balanced presentations. 
(See \cite{HNO'B} for a survey of what was known in the mid-1990s.)

The finite nilpotent 3-manifold groups
$Q(8k)\times\mathbb{Z}/a\mathbb{Z}$ (with  $(a,2k)=1$)
have the balanced presentations 
\[
\langle{x,y}\mid {x^{2ka}=y^2},~yxy^{-1}=x^s\rangle,
\]
where $s\equiv1~mod~(a)$ and $s\equiv-1~mod~(2k)$.
The other finite nilpotent groups $F$ with 4-periodic cohomology
(the generalized quaternionic groups 
$Q(2^na,b,c)\times\mathbb{Z}/d\mathbb{Z}$,
with $a,b,c,d$ odd and pairwise relatively prime) have $H_2(F;\mathbb{Z})=0$, 
but we do not know whether they all have balanced presentations.

\section{unipotent automorphisms}

An automorphism $\alpha$ of an abelian group $A$ is {\it unipotent\/} if
$\alpha-id_A$ is nilpotent.
The following lemma is a particular case of a result of P. Hall \cite[5.2.1]{Ro}.

\begin{lem}
[Hall]
Let $\psi$ be an automorphism of a finitely generated nilpotent group $N$.
Then $G=N\rtimes_\psi\mathbb{Z}$ is nilpotent if and only if $\psi^{ab}$ is unipotent.
\qed
\end{lem}

We shall extend the term ``unipotent",
to say that an automorphism $\psi$ of a finitely generated 
nilpotent group is unipotent if $\psi^{ab}$ is unipotent.
Furthermore, an action $\alpha:G\to{Aut}(A)$ is unipotent if
$\alpha(g)$ is unipotent for all $g\in{G}$.

\begin{lemma}
\label{filter}
Let $N$ be a finitely generated nilpotent group which acts unipotently
on a finitely generated abelian group $A$,
and let $\mathfrak{n}$ be the augmentation ideal of $\mathbb{Z}[N]$.
Then $A$ has a finite filtration $A=A_1>\dots>A_k=A^N>A_{k+1}=0$
by $\mathbb{Z}[N]$-submodules,
where $A^N$ is the fixed subgroup and $\mathfrak{n}A_i\leq{A_{i+1}}$, 
for $i\leq{k}$.
\end{lemma}

\begin{proof}
We induct on the length of the upper central series of $N$.
The centre $\zeta{N}$ is a nontrivial abelian group which acts unipotently
on $A$, and it is easy to see that $A^{\zeta{N}}\not=0$.
The quotient $N/\zeta{N}$ acts unipotently on each of $A^{\zeta{N}}$ and 
$\overline{A}=A/A^{\zeta{N}}$, and so these each have such filtrations,
by the inductive hypothesis. The preimages of the filtration
of $\overline{A}$ in $A$ combine with the filtration of $A^{\zeta{N}}$ 
to give the required filtration.
\end{proof}

It is easy to see that the product of commuting unipotent automorphisms
is unipotent.
This observation extends to show that an action of a nilpotent group 
$N$ is unipotent if $N$  is generated by elements which act unipotently.

Our next lemma is probably known,
but we have not found a published proof.

\begin{lemma}
\label{unipotent}
Let $\psi$ be a unipotent automorphism of a finitely generated nilpotent group $N$.
Then $H_i(\psi;R)$ and $H^i(\psi;R)$ are unipotent, 
for all simple coefficients $R$ and $i\geq0$.
\end{lemma}

\begin{proof}
If $N$ is cyclic then the result is clear.
In general,  $N$ has a composition series with cyclic
subquotients $\mathbb{Z}/p\mathbb{Z}$, where $p=0$ or is prime.
We shall induct on the number of terms in such a composition series.
If $N$ is infinite then $\psi$ acts unipotently on $Hom(N,\mathbb{Z})$
and so fixes an epimorphism to $\mathbb{Z}$;
if $N$ is finite then $\psi$ fixes an epimorphism to $\mathbb{Z}/p\mathbb{Z}$,
for any $p$ dividing the order of $N$.

Let $K$ be the kernel of such an epimorphism.
Then $\psi(K)=K$, by the choice of $\psi$;  let $\psi_K=\psi|_K$.
This is a unipotent automorphism of $K$, 
by Hall's Lemma \cite[5.2.10]{Ro}.
Hence the induced action of $\psi$ on $H_i(K;R)$ is unipotent,
for all $i$, by the inductive hypothesis.
Let $\Lambda=\mathbb{Z}[N/K]$ and let $B$ be a $\Lambda$-module.
Then $H_i(N/K;B)=Tor_i^\Lambda(\mathbb{Z},B)$ may be computed
from the tensor product $C_*\otimes_\mathbb{Z}B$,
where $C_*$ is a resolution of the augmentation $\Lambda$-module 
$\mathbb{Z}$.
If $B=H_i(K;R)$ then the diagonal action of $\psi$ on
each term of $C_*\otimes_\mathbb{Z}B$ is unipotent.
The result is now a straightforward consequence of the  
Lyndon-Hochschild-Serre spectral sequence for
$N$ as an extension of $N/K$ by $K$.

The argument for cohomology is similar.
\end{proof}

In fact we only need this lemma in degrees $\leq2$.
We shall usually assume that the coefficient ring is a field,
and then homology and cohomology are linear duals of each other.
Homology has an advantage deriving from the isomorphism
$G^{ab}\cong{H_1(G;\mathbb{Z})}$, 
but it is often more convenient to use cohomology instead.

\section{unipotent actions on abelian groups}

We shall find the following notion useful in many of our arguments.
Let  $G$ be a  group and $F$ a field.
Then an $F[G]$-module  $V$ is {\it canonically subsplit\/} if it contains a nontrivial direct sum of $F[G]$-submodules.
If $G$ acts unipotently on $V$ and $V$ is canonically subsplit then 
the subspaces of the summands fixed by $G/K$ are non-trivial,  
by Lemma \ref{filter}, 
and so the subspace $V^G$ fixed by $G$ has dimension $>1$.

\begin{lemma}
\label{cycad}
Let $A$ be a finitely generated abelian group and $p$ a prime
such that $A$ has non-trivial $p$-torsion and $\dim_{\mathbb{F}_p}A/pA>1$.
If $p$ is odd or if $p=2$ and $A$ has no $\mathbb{Z}/2\mathbb{Z}$ 
summand then $H_2(A;\mathbb{F}_p)$ and 
$H^2(A;\mathbb{F}_p)$ are each canonically subsplit with respect to the natural action of (subgroups of) $\mathrm{Aut}(A)$.
\end{lemma}

\begin{proof} 
Let $W=(A/pA)\wedge(A/pA)$ and 
$A^*=Hom (A;\mathbb{F}_p)=H^1(A;\mathbb{F}_p)$.

Then there is a natural splitting
$H_2(A;\mathbb{F}_p)=W\oplus{Tor(A,\mathbb{F}_p)}$
if $p$ is odd \cite[Chapter V.6]{Br},
or if $p=2$ and $A$ has no $\mathbb{Z}/2\mathbb{Z}$ summand \cite{IZ18}.
There is also a natural epimorphism
$\theta:H^2(A;\mathbb{F}_p)\to{Hom(W,\mathbb{F}_p)}$,
with kernel isomorphic to $Ext(A;\mathbb{F}_p)$ 
\cite[Exercises IV.3.8 and V.6.5]{Br}.

If $p$ is odd then cup product induces a monomorphism 
$c_A:A^*\wedge{A^*}\to{H^2(A;\mathbb{F}_p)}$, 
since $A$ is abelian.
If $p=2$ then cup product defines a homomorphism 
from $A^*\odot{A^*}$ to $H^2(A;\mathbb{F}_2)$.
Since $A$ has no $\mathbb{Z}/2\mathbb{Z}$ summand, 
$Sq(a)=a\cup{a}=0$ for all $a\in{A^*}$, 
and so cup product again induces a monomorphism 
$c_A:A^*\wedge{A^*}\to{H^2(A;\mathbb{F}_p)}$ \cite{Hi87}.
It is easily seen from the formulae in \cite{Br}
that $\theta\circ{c_A}$ is an isomorphism,
and so $H^2(A;\mathbb{F}_p)$ is naturally isomorphic to
$(A^*\wedge{A^*})\oplus{Ext(A;\mathbb{F}_p)}$.

The summands are all non-trivial,
since $A$ has nontrivial $p$-torsion  and $A/pA$ is not cyclic.
Thus $H_2(A;\mathbb{F}_p)$ and 
$H^2(A;\mathbb{F}_p)$ are each canonically subsplit.
\end{proof}

The case when $A$ has a summand of exponent 2 seems more complicated, 
and we consider only the cohomology.

\begin{lemma}
\label{cycad2}
Let $A$ be a finitely generated abelian group with a nontrivial summand of exponent $2$ and such that $\dim_{\mathbb{F}_2}A/2A>1$.
Suppose that a finitely generated nilpotent group $N$ acts unipotently on $A$.
Then $\dim_{\mathbb{F}_2}H^2(A;\mathbb{F}_2)^N>1$.
\end{lemma}

\begin{proof}
We may assume that $A\cong{B}\oplus{E}$,
where $E\cong(\mathbb{Z}/2\mathbb{Z})^s\not=0$
and $B$ has no summand of order 2. 
The subspace $B^*$ of $A^*=Hom(A,\mathbb{F}_2)=H^1(A;\mathbb{F}_2)$ 
consisting of homorphisms which factor through homomorphisms to 
$\mathbb{Z}/4\mathbb{Z}$ is canonical.
Clearly $B^*\cong{Hom(B,\mathbb{F}_2)}$ and 
$A^*/B^*\cong{E^*}=Hom(E,\mathbb{F}_2)$.
Hence $A^*\cong{B^*}\oplus {E^*}$, but this splitting is not canonical.
Cup product induces a homomorphism 
$c_A:A^*\odot{A^*}\to{H^2(A;\mathbb{F}_2)}$,
with kernel $2A/4A\cong{B^*}$,
since $A$ is abelian  \cite{Hi87}.
There is also a natural squaring map 
$Sq:A^*\to{H^2(A;\mathbb{F}_2)}$ with kernel $B^*$.

If $B=0$ then $A$ is an elementary 2-group and $A^*=E^*$,
and $c_A$ is a monomorphism.
Let $A_1>\dots>A_{k+1}=0$ be a filtration of $A^*$ by 
$\mathbb{F}_p[N]$-submodules, as in Lemma  \ref{filter}.
Then $A_k\odot{A_k}$ is fixed by $N$.
If $\dim_{\mathbb{F}_2}A_k>1$  then
$\dim_{\mathbb{F}_2}A_k\odot{A_k}\geq3$.
If $A_k$ has dimension 1,  and is generated by $b$ then
$b\odot{b}$ is fixed by $N$.
If $a\in{A_{k-1}}$ then each element of $N$ either fixes $a$ or 
sends it to $a+b$.
In either case $a\odot(a+b)$ is fixed by $N$.
Since $\dim_{\mathbb{F}_2}A_{k-1}\geq2$ the subspace
generated by $\{a\odot(a+b)\mid{a}\in{A_{k-1}}\}\cup\{b\odot{b}\}$
is fixed by $N$,
and so $\dim_{\mathbb{F}_2}H^2(A;\mathbb{F}_2)^N>1$.

The images of $B^*\odot{A^*}$ and $Sq(A^*)=Sq(E^*)$ are
canonical submodules of $H^2(A;\mathbb{F}_2)$,
with trivial intersection.
Hence they are invariant under the action 
of automorphisms of $A$,
and so if $B\not=0$ then we again have
$\dim_{\mathbb{F}_2}H^2(A;\mathbb{F}_2)^N>1$.
\end{proof}

\begin{cor}
\label{cycboth}
Let $A$ be a finitely generated abelian group, $\psi$ be  a unipotent
automorphism of $A$, and $p$ be a prime. 
If $A$ has non-trivial $p$-torsion and $\dim_{\mathbb{F}_p}A/pA>1$ 
then $\dim_{\mathbb{F}_p}\mathrm{Ker}(H_2(\psi;\mathbb{F}_p)-I)=
\dim_{\mathbb{F}_p}\mathrm{Ker}(H^2(\psi;\mathbb{F}_p)-I)>1$.
\end{cor}

\begin{proof}
Let $N$ be the cyclic subgroup of $Aut(A)$ generated by $\psi$.
We shall write $H_i(\psi)$ and $H^j(\psi)$ instead of 
$H_i(\psi;\mathbb{F}_p)$ and $H^j(\psi;\mathbb{F}_p)$,
for simplicity of notation.
Then $H_i(A;\mathbb{F}_p)^N=\mathrm{Ker}(H_i(\psi)-I)$
and $H^j(A;\mathbb{F}_p)^N=\mathrm{Ker}(H^j(\psi)-I)$, for any $i$.
If $\varphi$ is an endomorphism of a finite dimensional vector space $V$
then $\dim\mathrm{Cok}(\varphi)=\dim\mathrm{Ker}(\phi)$
and if $\varphi^*$ is the induced endomorphism of the dual vector space 
$V^*$ then $\varphi^*$ and $\varphi$ have the same rank.
Hence the corollary follows from Lemma \ref{cycad}, if $p$ is odd,
and from Lemma \ref{cycad2}, if $p=2$.
\end{proof} 

It does not seem obvious that 
$\dim_{\mathbb{F}_p}H_2(A;\mathbb{F}_p)^N$ and
$\dim_{\mathbb{F}_p}H^2(A;\mathbb{F}_p)^N$ are equal
when $N$ is not cyclic.

If $\dim_{\mathbb{F}_p}A/pA\geq4$ then the restriction of 
$H_2(\psi;\mathbb{F}_p)-I$ to $(A/pA)\wedge(A/pA)$ 
has kernel of dimension $>1$, 
and so $\dim_{\mathbb{F}_p}\mathrm{Ker}(H_2(\psi;\mathbb{F}_p)-I)>1$. 
In \cite{Hi22} a related observation for free abelian groups of rank 
$\geq4$ is used to show that if $G$ is a metabelian nilpotent group
with $h(G)>4$ then $\beta_2(G;\mathbb{Q})>\beta_1(G;\mathbb{Q})$.

One of the difficulties in extending the approach of this paper 
to more general nilpotent groups is the lack of  an analogue 
to the above lemmas for non-abelian $p$-groups.
If  $T$ is one of the 2-generator metacyclic $p$-groups of \S2 
then $H^2(T;\mathbb{F}_p)$ has no canonically split subspace, 
and such groups do arise as the torsion subgroups of homologically balanced nilpotent groups $G$ with $h(G)=1$.
(See the final paragraphs of \S5 below.)
Can we at least use such an argument to show that the torsion subgroup
must be homologically balanced?

\section{wang sequence estimates}

If $G$ is a finitely generated infinite nilpotent group then 
there is an epimorphism $f:G\to\mathbb{Z}$, 
and so $G\cong{K}\rtimes_\psi\mathbb{Z}$,
where $\psi$ is an automorphism of $K=\mathrm{Ker}(f)$ 
determined by conjugation in $G$.
The homology groups $H_i(K;R)=H_i(G;R[G/K])$ are $R[G/K]$-modules, 
with a generator $t$ of $G/K\cong\mathbb{Z}$ acting via $H_i(\psi;R)$.
The long exact sequence of homology
associated to the short exact sequence of coefficients
\[
0\to{R[G/K]}\xrightarrow{t-1}{R[G/K]}\to{R}\to0
\]
is the {\it Wang sequence} 
\[
H_2(K;R)\xrightarrow{H_2(\psi;R)-I}{H_2(K;R)}\to{H_2(G;R)}\to{H_1(K;R)}\xrightarrow{H_1(\psi;R)-I}{H_1(K;R)}\to
\]
\[\to{H_1(G;R)}\to{R}\to0.
\]
There is a similar Wang sequence for cohomology.
(These are special cases of the  Lyndon-Hochschild-Serre spectral sequences for the homology and cohomology with coefficients $R$ of $G$ as an extension of $\mathbb{Z}$ by $K$.)

\begin{lemma}
\label{wang app}
Let $G\cong{K}\rtimes_\psi\mathbb{Z}$ be a finitely generated nilpotent group,
and let $F$ be a field.
Then  
\begin{enumerate}
\item$\dim_F\mathrm{Cok}(H_2(\psi;F)-I)=
\dim_F\mathrm{Ker}(H^2(\psi;F)-I)=\beta_2(G)-\beta_1(G)+1$,
and so $\beta_2(G;F)\geq\beta_1(G;F)-1$,
with equality if and only if $\beta_2(K;F)=0$;
\item{}if $\beta_2(G;F)=\beta_1(G;F)$ then $H_2(K;F)$ is cyclic 
as a $F[G/K]$-module;
\item{}$\beta_1(G;F)=1\Leftrightarrow\beta_2(G;F)=0$, 
and then $K$ is finite,  $\beta_1(K;F)=0$, and $h(G)=1$;
\item{}if $H_2(G;\mathbb{Z})=0$ then $G\cong\mathbb{Z}$.
\end{enumerate}
\end{lemma}

\begin{proof}
Part (1) follows from the Wang sequences for the homology and cohomology of $G$ as an extension of $\mathbb{Z}$ by $K$.
The endomorphisms $H_i(\psi;F)-I$ have non-trivial kernel and cokernel if $H_i(K;F)\not=0$,  since they are nilpotent,  by Lemma \ref{unipotent}.

The $F[G/K]$-module $H=H_2(K;F)$ is  finitely generated 
and is annihilated by a power of $t-1$, 
since $H_2(\psi;F)$ is unipotent.
If $\beta_2(G;F)=\beta_1(G;F)$ then $\dim_FH/(t-1)H=1$,
by the exactness of the Wang sequence.
Since $F[G/K]\cong{F}[t,t^{-1}]$ is a PID,
it follows that $H$ is cyclic as an $F[G/K]$-module.

Let $t\in{G}$ represent a generator of $G/K$.
Then $F[G/K]\cong{F}[t,t^{-1}]$ is a PID
and $H=H_2(K;F)=H_2(G;F[G/K])$ is a finitely generated 
$F[t,t^{-1}]$-module, with $t$ acting via $H_2(\psi;F)$.
This module is annihilated by a power of $t-1$, 
since $H_2(\psi;F)$ is unipotent.
If $\beta_2(G;F)=\beta_1(G;F)$ then $\dim_FH/(t-1)H=1$,
by exactness of the Wang sequence.
Since $F[G/K]\cong{F}[t,t^{-1}]$ is a PID,
it follows that $H$ is cyclic as an $F[G/K]$-module.

If $\beta_1(G;F)=1$ then $H_1(K;F)=0$, and so $K$ is finite and $h(G)=1$.
Since $K$ is finite it is the direct product of its Sylow subgroups,
and the Sylow $p$-subgroup carries the $p$-primary homology of $K$.
Hence  if $F$ has characteristic $p>1$ and $H_1(K;F)=0$
then the Sylow $p$-subgroup is trivial and $H_i(K;F)=0$, for all $i\geq1$.
If $F$ has characteristic 0 then $H_i(K;F)=0$ for all $i\geq1$ also.
In each case, $H_i(G;F)=0$, for all $i>1$, and so $\beta_2(G;F)=0$.
Conversely, if $\beta_2(G;F)=0$ then $H_1(\psi;F)-I$ is a monomorphism.
Since $H_1(\psi;F)-I$ is nilpotent,  $H_1(K;F)=0$.
Hence $K$ is finite, so $h(G)=1$, and $\beta_1(G;F)=1$.

Part (4) is similar.
If $H_2(G;\mathbb{Z})=0$ then $\psi^{ab}-I$ is a monomorphism,
and so $K^{ab}=0$.
Hence $K=1$ and $G\cong\mathbb{Z}$.
\end{proof}

In particular, if $h(G)=1$ and $T$ is the torsion subgroup of $G$ then 
$\beta_1(T;\mathbb{F}_p)>0$ if and only if $\beta_1(G;\mathbb{F}_p)>1$. 
The fact that the torsion subgroup has non-trivial image in the abelianization
does not extend to nilpotent groups $G$ with $h(G)>1$,
as may be seen from the groups
with presentation $\langle{x,y}\mid[x,[x,y]]=[y,[x,y]]=[x,y]^p=1\rangle$.

\begin{cor}
\label{wang cor}
Let $G$ be a finitely generated nilpotent group.
Then
\begin{enumerate}
\item{} $\beta_2(G;\mathbb{Q})<\beta_1(G;\mathbb{Q})$ if and only if
$h(G)=1$ or $2$;
\item{}
if $\beta_2(G;\mathbb{F}_p)<\beta_1(G;\mathbb{F}_p)$ for some prime $p$ then $G$ is infinite, $G$ has no $p$-torsion and
$h(G)=\beta_1(G;\mathbb{F}_p)=1$ or $2$. 
\end{enumerate}
\end{cor}

\begin{proof}
If $G$ is finite then $\beta_i(G;\mathbb{Q})=0$ for all $i>0$,
and if $p$ divides the order of $G$ then it follows from the 
Universal Coefficient exact sequences of \S1 that
$\beta_2(G;\mathbb{F}_p)\geq\beta_1(G;\mathbb{F}_p)$,
since $_pG^{ab}$ and $G^{ab}/pG^{ab}$ have the same dimension.

Hence we may assume that $G$ is infinite,
and so $G\cong{K}\rtimes_\psi\mathbb{Z}$,
where $K$ is a finitely generated nilpotent group
and $\psi$ is a unipotent automorphism.
Let $F$ be a field.
We may use Lemma \ref{wang app} to show first that
$\beta_2(K;F)=0$ and then that $\beta_1(K;F)\leq1$.
Hence $\beta_1(G;F)\leq2$.

If $F=\mathbb{Q}$ then either $K$ is finite and $h(G)=1$, 
or $h(K)=1$ and $h(G)=2$.
The converse is clear, since $G$ is then a finite extension of 
$\mathbb{Z}^{h(G)}$.

Suppose that $F=\mathbb{F}_p$ for some prime $p$.
If $\beta_1(G;\mathbb{F}_p)=1$ then $K$ is finite,
so $h(G)=1$, and $\beta_2(K;\mathbb{F}_p)=0$, 
so $\beta_1(K;\mathbb{F}_p)=0$ and $K$ has no $p$-torsion.
If $\beta_1(G;\mathbb{F}_p)=2$ then $\beta_1(K;\mathbb{F}_p)=1$
and $\beta_2(K;\mathbb{F}_p)=0$, 
so $h(K)=1$ and $K$ has no $p$-torsion.
Hence $h(G)=2$ and $G$ has no $p$-torsion.
\end{proof}

The next result is a corrected version of Lemma 1 of \cite{Hi22}
(in which it was inadvertently assumed that $\beta_2(G)=\beta_1(G)$
if $G$ is homologically balanced).

\begin{lemma}
\label{h=2}
Let $G$ be a finitely generated nilpotent group 
and let $\beta=\beta_1(G;\mathbb{Q})$.
Then $G$ is homologically balanced if and only if 
$H_2(G;\mathbb{Z})$ is a quotient of $\mathbb{Z}^\beta$;
if $h(G)>2$ then $G$ is homologically balanced if and only if 
$H_2(G;\mathbb{Z})\cong\mathbb{Z}^\beta$.
\end{lemma}

\begin{proof}
Since $G^{ab}\cong\mathbb{Z}^\beta\oplus{B}$, where $B$ is finite,
the first assertion follows from the Universal Coefficient exact sequences of \S1.
If $h(G)>2$ then $\beta_2(G;\mathbb{Q})\geq\beta$, 
by Corollary \ref{wang cor}, 
and so  $H_2(G;\mathbb{Z})$ is a quotient of $\mathbb{Z}^\beta$  
if and only if $H_2(G;\mathbb{Z})\cong\mathbb{Z}^\beta$.
\end{proof}

\section{$h=1$: virtually $\mathbb{Z}$}

We include the following simple lemma as some of the observations 
are not explicit in our primary reference \cite{Ro}.

\begin{lemma}
\label{2end}
Let $G$ be a finitely generated nilpotent group,
and let $T$ be its torsion subgroup. 
Then the following are equivalent
\begin{enumerate}
\item$\beta_1(G;\mathbb{Q})=1$;
\item$h(G)=1$;
\item $G/T\cong\mathbb{Z}$;
\item$G\cong{T}\rtimes_\psi\mathbb{Z}$, 
where $\psi$ is an automorphism of $T$.
\end{enumerate}
\end{lemma}

\begin{proof}
In each case $G$ is clearly infinite, 
and so there is an epimorphism $f:G\to\mathbb{Z}$, 
with kernel $K$, say.
Since $G$ is finitely generated, so is $K$.
If $\beta_1(G;\mathbb{Q})=1$ then $K$ is finite, by Lemma \ref{wang app}.
If $h(G)=1$ then $h(K)=0$, so $K$ is again finite.
In each case, $K=T$ and $G/T\cong\mathbb{Z}$.
If $G/T\cong\mathbb{Z}$ and $t\in{G}$ represents a generator
of $G/T$ then conjugation by $t$ defines an automorphism $\psi$
of $T$, and $G\cong{T}\rtimes_\psi\mathbb{Z}$.
Finally, it is clear that (4) implies each of (1) and (2).
\end{proof}

We could also describe the groups considered on this lemma as
the nilpotent groups which are virtually $\mathbb{Z}$,
and as the nilpotent groups with two ends.

In the next lemma we do not assume that $G$ is nilpotent.

\begin{lemma}
\label{hbh1}
Let $G\cong{T}\rtimes_\psi\mathbb{Z}$ be a homologically balanced group, where $T$ is finite.
Then $H_2(G;\mathbb{Z})$ is a finite cyclic group, and 
$|H_2(G;\mathbb{Z})|$ is divisible by the order of the torsion subgroup of $G^{ab}$.
\end{lemma}

\begin{proof}
It is immediate from the Wang sequence for the integral homology 
of $G$ as an extension of $\mathbb{Z}$ by $T$ that 
$H_2(G;\mathbb{Z})$ is finite.
It is also clear that $C=\mathrm{Cok}(\psi^{ab}-I)$ is
the torsion subgroup of $G^{ab}$.
Since $T$ is finite,
$|\mathrm{Ker}(\psi^{ab}-I)|=|\mathrm{Cok}(\psi^{ab}-I)|$,
and so $|C|$ divides $|H_2(G;\mathbb{Z})|$.

If $p$ is a prime then 
$\dim_{\mathbb{F}_p}Tor(\mathbb{F}_p,G^{ab})=
\dim_{\mathbb{F}_p}Tor(\mathbb{F}_p,C)=
\beta_1(G;\mathbb{F}_p)-1$,
since $G^{ab}\cong\mathbb{Z}\oplus{C}$.
Therefore 
\[
\dim_{\mathbb{F}_p}{Hom}(H_2(G;\mathbb{Z}),\mathbb{F}_p)=
\beta_2(G;\mathbb{F}_p)-\beta_1(G;\mathbb{F}_p)+1,
\]
by the Universal Coefficient exact sequences of \S1.
Since $G$ is homologically balanced,
this is at most 1, for all primes $p$, 
and so $H_2(G;\mathbb{Z})$ is cyclic.
\end{proof}

If $G$ is nilpotent then $H_2(\psi)-I$ is a nilpotent endomorphism of 
$H_2(T;\mathbb{Z})$, 
and so $|H_2(G;\mathbb{Z})|=|C|$ if and only if $H_2(T;\mathbb{Z})=0$.

\begin{theorem}
\label{h=1}
Let $G\cong{T}\rtimes_\psi\mathbb{Z}$, 
where $T$ is a finite nilpotent group and $\psi$ is a unipotent automorphism of $T$.
If $G$ is  homologically balanced then
\begin{enumerate}
\item{}$G$ is $2$-generated;
\item{} if the Sylow $p$-subgroup of $T$ is abelian then it is cyclic;
\item{}if $T$ is abelian then 
$G\cong\mathbb{Z}/m\mathbb{Z}\rtimes_n\mathbb{Z}$, 
for some $m,n\not=0$ such that $m$ divides a power of $n-1$.
\end{enumerate}
Conversely, if $T$ is homologically balanced and $G$is $2$-generated then $G$
is homologically balanced.
\end{theorem}

\begin{proof}
Let $p$ be a prime.
Then $\dim_{\mathbb{F}_p}H_1(T;\mathbb{F}_p)=
\dim_{\mathbb{F}_p}Tor(T^{ab},\mathbb{F}_p)$,
since $T$ is finite.
Moreover,
$\psi^{ab}-I$ and $Tor(\psi^{ab},\mathbb{F}_p)-I$ have the same rank.
Since $Tor(T^{ab},\mathbb{F}_p)$ is a natural quotient of 
$H_2(T;\mathbb{F}_p)$,  exactness of the Wang sequence implies
that $\beta_2(G;\mathbb{F}_p)\geq2(\beta_1(G;\mathbb{F}_p)-1)$.
Since $G$ is homologically balanced, $\beta_1(G;\mathbb{F}_p)\leq2$.
Hence the $p$-torsion of $G^{ab}$ is cyclic.
Therefore $G^{ab}\cong\mathbb{Z}\oplus{C}$ for some finite cyclic group.
Since $G$ is nilpotent and $G^{ab}$ is 2-generated, so is $G$.

The Sylow subgroups of $T$ are characteristic,
and $\psi$ restricts to a unipotent automorphism of each such subgroup.
Suppose that the Sylow $p$-subgroup of $T$ is an abelian group $A$.
Since $H^2(A;\mathbb{F}_p)=Hom(H_2(A;\mathbb{F}_p),\mathbb{F}_p)$,
the endomorphisms $H^2(\psi;\mathbb{F}_p)-I$ and 
$H_2(\psi;\mathbb{F}_p)-I$ 
have the same rank.
Hence
\[
\dim(\mathrm{Ker}(H^2(\psi;\mathbb{F}_p)-I)=
\dim(\mathrm{Ker}(H_2(\psi;\mathbb{F}_p)-I)=
\dim(\mathrm{Cok}(H_2(\psi;\mathbb{F}_p)-I)\leq1.
\]
Hence $A$ must be cyclic, by Corollary \ref{cycboth}.

It follows immediately that if $T$ is abelian then it is a direct product 
of cyclic groups of relatively prime orders,
and so is cyclic, of order $m$, say.
Hence $G\cong\mathbb{Z}/m\mathbb{Z}\rtimes_n\mathbb{Z}$, 
for some $n$ such that $(m,n)=1$.
Such a semidirect product is nilpotent if and only if 
$m$ divides some power of $n-1$.

The final assertion follows from consideration of the Wang sequences with coefficients $\mathbb{F}_p$, for $p$ dividing the order of $T$.
\end{proof}

If $G$ is homologically balanced must $T$ also be homologically balanced?

Every semidirect product $\mathbb{Z}/m\mathbb{Z}\rtimes_n\mathbb{Z}$ 
has a balanced presentation
\[
\langle{a,t}\mid{a^m=1},~tat^{-1}=a^n\rangle.
\]
The simplest examples with $T$ non-abelian are the groups 
$Q(8k)\rtimes\mathbb{Z}$,
with the balanced presentations 
$\langle{t,x,y}\mid{x^{2k}=y^2},~tx=xt,~tyt^{-1}=xy\rangle$,
which simplify to
\[
\langle{t,y}\mid[t,y]^{2k}=y^2,~[t,[t,y]]=1\rangle.
\]
Let $m=p^s$ , where $p$ is a prime and $s\geq1$, and let $G$ be the group with presentation
\[
\langle{t,x,y}\mid{txt^{-1}=y},~tyt^{-1}=x^{-1}y^2,~yxy^{-1}=x^{m+1}\rangle.
\]
If we conjugate the final relation with $t$ to get the relation 
$x^{-1}yx=y^{m+1}$ then we see that the torsion subgroup $T$
has presentation $\langle{x,y}\mid{x^m=y^m},~yxy^{-1}=x^{m+1}\rangle$,
and so is one of the metacyclic $p$-groups mentioned at the end of \S2.
Moreover,  $G$ is nilpotent, 
$\zeta{G}=\langle{x^m}\rangle$ and $G'=\langle{x^m,x^{-1}y}\rangle$ is abelian.
Hence $G$ is metabelian.
Each of the groups that we have described here
is 2-generated and its torsion subgroup is homologically balanced. 

\section{metabelian nilpotent groups with hirsch length $>1$}

All known examples of nilpotent groups with balanced presentations 
and Hirsch length $h>1$ are torsion-free. 
We have not yet been able to show that this must be so.
However,  if such a group is also metabelian, but not $\mathbb{Z}^3$,  then 
$h(G)\leq4$ and $\beta_1(G;\mathbb{Q})=2$ \cite[Theorems 7 and 15]{Hi22}.
Our main result implies that there are just three such groups 
with $G/G'\cong\mathbb{Z}^2$.
The argument again involves finding normal subgroups with
``large enough" homology to affect the Betti numbers of the group.
We develop a number of lemmas to this end.

\begin{lemma}
\label{h=2L}
Let $G$ be a finitely generated nilpotent group,
and let $T$ be its torsion subgroup. 
Then the following are equivalent
\begin{enumerate}
\item$\beta_1(G;\mathbb{Q})=2$ and $\beta_2(G;\mathbb{Q})=1$;
\item$h(G)=2$;
\item$G/T\cong\mathbb{Z}^2$.
\end{enumerate}
If these conditions hold and $G$ is homologically balanced then
$H_2(G;\mathbb{Z})\cong\mathbb{Z}\oplus\mathbb{Z}/e\mathbb{Z}$,
for some $e\geq1$.
\end{lemma}

\begin{proof}
If (1) holds then $h(G)\geq2$,  and so $h(G)=2$,  
by the corollary to Lemma \ref{wang app}.
It is easy to see that (2) and (3) are equivalent, and imply (1).
The final assertion follows from Corollary \ref{wang cor} and Lemma \ref{h=2}.
\end{proof}

We could also describe the groups considered in this lemma as
the nilpotent groups which are virtually $\mathbb{Z}^2$.

In this section the Lyndon-Hochschild-Serre spectral sequences for the homology and cohomology of a group which is an extension of $\mathbb{Z}^2$
by a normal subgroup 
shall largely replace the Wang sequences used above.

\begin{lemma}
\label{euler}
Let $F$ be a field and $A$ be a finite dimensional $F[\mathbb{Z}^2]$-module,
and let $b_i=\dim_FH_i(\mathbb{Z}^2;A)$, for $i\geq0$.
Then $b_2=b_0$ and $b_1=b_0+b_2=2b_0$.
\end{lemma}

\begin{proof}
We may compute $H_i(\mathbb{Z}^2;A)=Tor_i^{F[\mathbb{Z}^2]}(F,A)$ 
from the complex
\[
0\to{A}\to{A^2}\to{A}\to0,
\]
in which the differentials are $\partial^1=
\left[\smallmatrix(x-1)id_A\\(y-1)id_A\endsmallmatrix\right]$,
and $\partial^2=\left[\smallmatrix(y-1)id_A,(1-x)id_A\endsmallmatrix\right]$
where $\{x,y\}$ is a basis for $\mathbb{Z}^2$.
Since the matrix for $\partial^2$ is the transpose of that for
$\partial^1$ (up to a change of sign in the second block),
they have the same rank.
Hence 
$b_2=\dim_F\mathrm{Ker}(\partial_2)=\dim_F\mathrm{Cok}(\partial_1)=b_0$.
The final assertion follows since $b_0-b_1+b_2=1-2+1=0$ is the Euler characteristic of the complex.
\end{proof}

The  modules $H_2(\mathbb{Z}^2;A)$ and $H_0(\mathbb{Z}^2;A)$ 
are the submodule of fixed points and the coinvariant quotient modules 
of the $\mathbb{Z}^2$-action, respectively.
Minor adjustments give similar results for $\dim_FH^j(\mathbb{Z}^2;A)$.
(We may also use Poincar\'e duality for $\mathbb{Z}^2$
to relate homology and cohomology.)

Recall that if $K$ is a normal subgroup of a group $G$ then conjugation in $G$ induces a natural action of $G/K$ on the homology and cohomology of $K$.

\begin{lemma}
\label{metab1}
Let $G$ be a finitely generated nilpotent group 
with a normal subgroup $K$ such that $G/K\cong\mathbb{Z}^2$,
and suppose that $\beta_2(G;F)\leq\beta_1(G;F)$ for some field $F$.
Then $\dim_FH^2(K;F)^{G/K}\leq1$.
\end{lemma}

\begin{proof}
We note first that $\beta_1(G;F)=2$ or 3,
since $\beta_2(G;F)\leq\beta_1(G;F)$ 
\cite[Theorem 2.7]{Lu83}.
We may assume that  $A=H^1(K;F)\not=0$,
since $G$ is nilpotent. 
Let $N=G/K$ and $b_i=\dim_FH^i(N;A)$.
The LHS spectral sequence for cohomology with coefficients $F$
for $G$ as an extension of $N$ by $K$ 
gives two exact sequences
\begin{equation*}
\begin{CD}
0\to{H^1(N;F)}\to{H^1(G;F)}\to{A}^N@>d^{0,1}_2>>{H^2(N;F)}\to
{H^2(G;F)}\to{J}\to0
\end{CD}
\end{equation*}
and
\begin{equation*}
\begin{CD}
0\to{H^1(N;A)}\to{J}\to{H^2(K;F)^N}@>d^{0,2}_2>>{H^2(N;A)}.
\end{CD}
\end{equation*}
The first sequence gives $\dim_FJ\leq\beta_2(G;F)$.
Then  $b_1=b_0+b_2=2b_0$, by Lemma \ref{euler}, and $b_0>0$,
since $A\not=0$.
Hence $b_0-b_1<0$, and so the second sequence gives   
$\dim_FH^2(K;F)^N\leq
\beta_2(G;F)+b_0-b_1\leq\beta_2(G;F)-1$.
In particular, 
$\dim_FH^2(K;F)^N\leq1$ if $\beta_2(G;F)=2$.

If $\beta_2(G;F)=3$ then
$\beta_1(G;F)=3$ also, by Corollary \ref{wang cor}
and so $\dim_F\mathrm{Ker}(d^{0,1}_2)=1$.
If $d^{0,1}_2\not=0$ then $b_0=2$ and so $b_1=4$,
by Lemma \ref{euler}.
But then $\beta_2(G;F)\geq4$.
Therefore $d^{0,1}_2=0$, and so $b_0=1$.
Hence $b_1=2$ and $b_2=1$,
and $d^{0,2}_2$ is a monomorphism.
Hence we again have $\dim_FH^2(K;F)^N\leq1$.
\end{proof}

In particular, $H^2(K;F)$ is not canonically subsplit.

A parallel argument using the LHS spectral sequence for homology shows that
$\dim_FH_0(G/K;H_2(K;F))\leq1$.

\begin{lemma}
\label{K 2end}
Let $P$ be a non-trivial finite $p$-group and $K\cong\mathbb{Z}\times{P}$.
Then $H^2(K;\mathbb{F}_p)$ is canonically subsplit.
\end{lemma}

\begin{proof}
We shall use the Universal Coefficient exact sequence for cohomology given in \S1.
The projection of $K$ onto $K/P\cong\mathbb{Z}$ determines 
a class $\eta\in{H^1(K;\mathbb{F}_p)}=Hom(K,\mathbb{F}_p)$ (up to sign),
and cup product with $\eta$ maps $H^1(K;\mathbb{F}_p)$ non-trivially to 
$H^2(K;\mathbb{F}_p)$, by the K\"unneth Theorem.
The restriction from $Ext(K^{ab},\mathbb{F}_p)$ to 
$Ext(P^{ab},\mathbb{F}_p)$ is an isomorphism, 
and so  $Ext(K^{ab},\mathbb{F}_p)$ and 
$\eta\cup{H^1(K;\mathbb{F}_p)}$ have trivial intersection.
Hence 
$Ext(K^{ab},\mathbb{F}_p)\oplus(\eta\cup{H^1(K;\mathbb{F}_p)})$ 
is a subspace of $H^2(K;\mathbb{F}_p)$,
and the summands  are invariant under 
the action of automorphisms of $K$,
by the naturality of the Universal Coefficient Theorem.
The summands are non-trivial, since $P\not=1$.
\end{proof}

The next four lemmas (leading up to Theorem \ref{Z^2}) consider 
nilpotent groups which are extensions of $\mathbb{Z}^2$ 
by finite normal subgroups.

\begin{lemma}
\label{P factor}
Let $G$ be a finitely generated nilpotent group,
and let $T$ be its torsion subgroup.
Let $P$ be a non-trivial Sylow $p$-subgroup of $T$ and let 
$\gamma_p:G\to{Out}(P)$
be the homomorphism determined by conjugation in $G$.
If $G/T\cong\mathbb{Z}^2$ and the image of $\gamma_p$ is cyclic then 
$\beta_2(G;\mathbb{F}_p)>\beta_1(G;\mathbb{F}_p)$.
\end{lemma}

\begin{proof}
We may write $G\cong{K}\rtimes_\psi\mathbb{Z}$, 
where $\psi$ is a unipotent automorphism of $K$,
and $K$ is in turn an extension of $\mathbb{Z}$ by $T$.
Let $P$ be the Sylow $p$-subgroup of $T$,
and let $N$ be the product of the other Sylow subgroups of $T$.
Since the Sylow subgroups of $T$ are characteristic, 
conjugation in $G$ determines a homomorphism $\gamma_p:G\to{Out}(P)$.
Moreover, $N$ is normal in $G$,  and the projection of $G$ onto $G/N$ 
induces isomorphisms on homology and cohomology 
with coefficients $\mathbb{F}_p$.
Hence we may assume that $N=1$ and so $T=P$ is a non-trivial $p$-group.

If the image of $\gamma_P$ is cyclic then $\gamma_P$ 
factors through an epimorphism $f:G\to\mathbb{Z}$, 
with kernel $K\cong\mathbb{Z}\times{P}$.
Since $H^2(K;\mathbb{F}_p)$ has a subspace which is the direct sum of non-trivial canonical summands, by Lemma \ref{K 2end},
 $\dim_{\mathbb{F}_p}\mathrm{Ker}(H^2(\psi;\mathbb{F}_p)-I)>1$
(as in Lemma \ref{cycad}).
The result now follows from Lemma \ref{wang app}.
\end{proof}

Thus the group with presentation 
$\langle{x,y}\mid[x,[x,y]]=[y,[x,y]]=[x,y]^p=1\rangle$
mentioned near the end of \S1 above does not have a balanced presentation.
Similarly, no nilpotent extension of $\mathbb{Z}^2$ by $Q(8)$ can have a balanced presentation,
since the abelian subgroups of $Out(Q(8))\cong\mathcal{S}_3$ are cyclic.

If $p$ is an odd prime and $C$ is a cyclic $p$-group then 
$\mathrm{Aut}(C)$ is cyclic, and so Lemma \ref{P factor} may apply.
However, dealing with 2-torsion again requires more effort.

\begin{lemma}
\label{partial3}
Let $G$ be a finitely generated nilpotent group,
and let $T$ be its torsion subgroup.
If $G/T\cong\mathbb{Z}^2$ and the Sylow $2$-subgroup of $T$
is a nontrivial cyclic group then 
$\beta_2(G;\mathbb{F}_2)>\beta_1(G;\mathbb{F}_2)$. 
\end{lemma}

\begin{proof}
We may factor out the maximal odd-order subgroup of $T$ without changing
the $\mathbb{F}_2$-homology, 
and so we may assume that $T\cong\mathbb{Z}/k\mathbb{Z}$,
where $k=2^n$,
for some $n\geq1$.
We may also assume that the action of $G$ on $T$ 
by conjugation does not factor through a cyclic group, 
by Lemma \ref{P factor}, and so $k\geq8$.
Let $U$ be the subgroup of $(\mathbb{Z}/k\mathbb{Z})^\times$
represented by integers $\equiv1~mod~(4)$.
Then $Aut(\mathbb{Z}/k\mathbb{Z})\cong\{\pm1\}\times{U}$.
It is easily verified that noncyclic subgroups of $Aut(\mathbb{Z}/k\mathbb{Z})$
have $\{\pm1\}$ as a direct factor, and so $G$ has a presentation
\[
\langle{x,y,z}\mid[x,y]=z^f,~z^k=1,~xzx^{-1}=z^{-1},~yzy^{-1}=z^\ell\rangle,
\]
where $f$ divides $k$, 
$1<\ell<k$ and $\ell\equiv1~mod~(4)$.
Let $m$ be a mutiplicative inverse for $\ell~mod~(k)$,
so that $1<m<k$ and $m\ell=wk+1$ for some $w\in\mathbb{Z}$.
Note that $\beta_1(G;\mathbb{F}_2)=2$ if $f=1$ and is 3 if $f>1$.

The ring $\mathbb{Z}[G]$ is a twisted polynomial extension 
of $\mathbb{Z}[\mathbb{Z}/k\mathbb{Z}]=\mathbb{Z}[z]/(z^k-1)$, 
and so is  noetherian.
We may assume each monomial is normalized in alphabetical order: 
$x^hy^iz^j$, for exponents $h,i\in\mathbb{Z}$ and $0\leq{j}<k$.
Let $\nu=\Sigma_{i=0}^{k-1}z^i$ be the norm element for 
$\mathbb{Z}[\mathbb{Z}/k\mathbb{Z}]$.
Then $z\nu=\nu$, so $\nu^2=k\nu$, and $\nu$ is central in $\mathbb{Z}[G]$.
We shall use the fact that if $\gamma,\delta\in\mathbb{Z}[G]$
are such that $\gamma\nu=0$ and $\delta(z-1)=0$
then $\gamma=\gamma'(z-1)$ and $\delta=\delta'\nu$, 
for some $\gamma',\delta'\in\mathbb{Z}[G]$.
On the other hand, non-zero terms not involving $z$
are not zero-divisors in $\mathbb{Z}[G]$.

The augmentation module $\mathbb{Z}$ has a Fox-Lyndon partial resolution
\begin{equation*}
\begin{CD}
C_2@>\partial_2>>{C_1}@>\partial_1>>C_0=\mathbb{Z}[G]@>\varepsilon>>\mathbb{Z}\to0,
\end{CD}
\end{equation*}
where $\varepsilon:\mathbb{Z}[G]\to\mathbb{Z}$ is the augmentation
homomorphism,
$C_1\cong\mathbb{Z}[G]^3$ has basis $\{e_x,e_y,e_z\}$
corresponding to the generators and 
 $C_2\cong\mathbb{Z}[G]^4$ has basis $\{r,s,t,u\}$
corresponding to the relators $r=z^fyxy^{-1}x^{-1}$, 
$s=z^k$, $t=zxzx^{-1}$ and $u=z^\ell{y}z^{-1}y^{-1}$.
The differentials are given by
\[
\partial_1(e_x)=x-1, \quad \partial_1(e_y)=y-1\quad\mathrm{ and}\quad 
\partial(e_z)=z-1;\quad\mathrm{ and}
\]
\[
\partial_2(r)=(z^fy-1)e_x+(z^f-x)e_y+(\Sigma_{i=0}^{f-1}z^i )e_z,\quad
\partial_2(s)=\nu{e_z},
\]
\[
\partial_2(t)=(z-1)e_x+(1+zx)e_z\quad\mathrm{and}\quad
\partial_2(u)=(z^\ell-1)e_y+(\Sigma_{j=0}^{\ell-1}z^j-y)e_z.
\]
We may choose a homomorphism $\partial_3:C_3\to{C_2}$ with domain 
$C_3$ a  free $\mathbb{Z}[G]$-module and image $\mathrm{Ker}(\partial_2)$, 
which extends the resolution one step to the left.
(We may assume that $C_3$ is finitely generated, 
since $\mathbb{Z}[G]$ is noetherian.)
It is clear from the Fox-Lyndon partial resolution that 
$\dim_{\mathbb{F}_2}\mathrm{Ker}(\mathbb{F}_2\otimes_{\mathbb{Z}[G]}\partial_2)
=\beta_1(G;\mathbb{F}_2)+1$.
We shall show that $\mathbb{F}_2\otimes_{\mathbb{Z}[G]}\partial_3=0$,
and so $\beta_2(G;\mathbb{F}_2)=\beta_1(G;\mathbb{F}_2)+1$.

Let $\varepsilon_2:\mathbb{Z}[G]\to\mathbb{F}_2$ be the $mod~(2)$
reduction of $\varepsilon$.
Since $\mathrm{Im}(\partial_3)=\mathrm{Ker}(\partial_2)$,
it shall suffice to show that if
\[
\partial_2(ar+bs+ct+du)=0,
\]
for some $a,b,c,d\in\mathbb{Z}[G]$ then 
$\varepsilon_2(a)=\varepsilon_2(b)=
\varepsilon_2(c)=\varepsilon_2(d)=0$.

The coefficients $a,b,c,d$ must satisfy the three equations
\[
a(z^fy-1)+c(z-1)=0,
\]
\[
a(z^f-x)+d(z^\ell-1)=0
\]
and
\[
a(\Sigma_{i=0}^{f-1}z^i )+b\nu+c(zx+1)+d(\Sigma_{j=0}^{\ell-1}z^j-y)=0.
\]
Multiplying the first of these equations by $\nu$ gives $af\nu(y-1)=0$.
Hence $a\nu=0$ and so $a=A(z-1)$, 
for some $A\in\mathbb{Z}[G]$ not involving $z$.
The first equation becomes
\[
A(z-1)(z^fy-1)+c(z-1)=[A(yz^{fm}(\Sigma_{j=0}^{m-1}z^j)-1)+c](z-1)=0,
\]
and so $c=-A(yz^{fm}(\Sigma_{j=0}^{m-1}z^j)-1)+C\nu$,
for some $C\in\mathbb{Z}[G]$ not involving $z$.
Similarly, the second equation becomes
\[
A(z-1)(z^f-x)+d(z^\ell-1)=A(zx+z^f)(z^{m\ell}-1)+d(z^\ell-1)=0,
\]
and so $d=-A(zx+z^f)(\Sigma_{j=0}^{m-1}z^{j\ell})+D\nu$, 
for some $D\in\mathbb{Z}[G]$ not involving $z$.
At this point it  is already clear that $\varepsilon_2(a)=
\varepsilon_2(c)=\varepsilon_2(d)=0$.

Multiplying the third equation by $\nu$ gives
\[
kb\nu+c\nu(x+1)+d\nu(\ell-y)=0.
\]
Rearranged and written out in full, this becomes
\[
kb\nu=(A(ym-1)-Ck)(x+1)\nu+(A(x+1)m-Dk)(\ell-y)\nu.
\]
Since $yx=z^{-f}xy=xzy=xyz^{fm}$ we have $yx\nu=xy\nu$
and so this simplifies to 
\[
kb\nu=(A(m\ell-1)(x+1)-kC(x+1)-kD(\ell-y))\nu.
\]
Write $b=b_1+B(z-1)$, where $b_1$ does not involve $z$.
Then $b\nu=b_1\nu$.
Since the terms $b_1, A,C$ and $D$ do not involve $z$,
and since $m\ell-1=wk$, we get
\[
kb_1=k(Aw(x+1)-C(x+1)-D(\ell-y)).
\]
We may solve for $b_1$,  and so
\[
b=b_1+B(z-1)=wA(x-1)+B(z-1)-C(x+1)-D(\ell-y).
\]
Hence $\varepsilon_2(b)=0$ also, 
so $\mathbb{F}_2\otimes_{\mathbb{Z}[G]}\partial_3=0$ and thus
$\beta_2(G;\mathbb{F}_2)=\beta_1(G;\mathbb{F}_2)+1$.
\end{proof}

\begin{lemma}
\label{metab}
Let $G$ be a finitely generated nilpotent group,
and let $T$ be its torsion subgroup. 
If $G/T\cong\mathbb{Z}^2$ and the Sylow $p$-subgroup of $T$ 
is abelian and non-trivial then 
$\beta_2(G;\mathbb{F}_p)>\beta_1(G;\mathbb{F}_p)$. 
\end{lemma}

\begin{proof}
Let $N$ be the product of all the Sylow $p'$-subgroups of $T$ with $p'\not=p$,
and let $A$ be the image of $T$ in $\overline{G}=G/N$.
Then $\beta_i(\overline{G};\mathbb{F}_p)=\beta_i(G;\mathbb{F}_p)$, 
for all $i$.
The Sylow $p$-subgroup of $T$ projects isomorphically onto $A$, 
and $\overline{G}/A\cong\mathbb{Z}^2$.
If $\beta_2(G;\mathbb{F}_p)\leq\beta_1(G;\mathbb{F}_p)$
then $\dim_{\mathbb{F}_p}H^2(A;\mathbb{F}_p)^{G/A}\leq1$,
by Lemma \ref{metab1}.
Hence $A$ is cyclic, by Lemmas \ref{cycad} and \ref{cycad2}.
If $p=2$ the result follows from Lemma \ref{partial3},
while if $p$ is odd it follows from Lemma \ref{P factor},
since the automorphism group of a cyclic group of odd $p$-power order is cyclic.
\end{proof}

For the next result we need an analogue of Lemma \ref{K 2end}.

\begin{lemma}
\label{K 1end}
Let $K\cong{T}\rtimes\mathbb{Z}^2$, 
where $T$ is a finite $p$-group,
and $T\not\leq{K'}$.
Then $H^2(K;\mathbb{F}_p)$ is canonically subsplit.
\end{lemma}

\begin{proof}
Let $\alpha:K\to\mathbb{Z}^2$ be the canonical epimorphism.
Since $\alpha$ splits,
$H^2(\alpha;\mathbb{F}_p)$ is a monomorphism.
The other hypotheses imply $Ext(K^{ab},\mathbb{F}_p)\not=0$.
Hence
$Ext(K^{ab},\mathbb{F}_p)\oplus\mathrm{Im}(H^2(\alpha;\mathbb{F}_p))$ 
is a subspace of $H^2(K;\mathbb{F}_p)$ with the desired properties.
\end{proof}

If $G$ is a homologically balanced, 
metabelian nilpotent group then either $G\cong\mathbb{Z}^3$ 
or $\beta_1(G;\mathbb{Q})\leq2$ and $h(G)\leq4$ \cite{Hi22}.
In the latter case the torsion-free quotient $G/T$ is either free abelian 
of rank $\leq2$, 
or is a $\mathbb{N}il^3$-group $\Gamma_q$ with presentation
\[
\langle{x,y,z}\mid[x,y]=z^q,~xz=zx,~yz=zy\rangle,
\]
or is the $\mathbb{N}il^4$-group $\Omega$ with  presentation
\[
\langle{t,u}\mid[t,[t,[t,u]]]=[u,[t,u]]=1\rangle.
\]
(See \cite[Corollary 8 and Theorems 10 and 15]{Hi22}.)

\begin{theorem}
\label{Z^2}
Let $G$ be a homologically balanced nilpotent group with 
$\beta_1(G;\mathbb{Q})=2$,
and let $T$ be its torsion subgroup.
Then
\begin{enumerate}
\item{}if $h(G)=2$ and $T$ is abelian then $G\cong\mathbb{Z}^2$; 
\item{}if $h(G)=3$  and the outer action $:G\to{Out(T)}$ determined by
conjugation in $G$ factors through $\mathbb{Z}^2$ then $G\cong\Gamma_q$,
for some $q\geq1$; 
\item{}if $h(G)=4$ and $G$ has an abelian normal subgroup $A$ with
$G/AT\cong\mathbb{Z}^2$ then $G\cong\Omega$.
\end{enumerate}
\end{theorem}

\begin{proof}
Let $K$ be the preimage in $G$ of the torsion subgroup of $G^{ab}$.
Then $TG'\leq{K}$, $G/K\cong\mathbb{Z}^2$, and $K/G'$ is (finite) cyclic, 
since $\beta_1(G;\mathbb{F}_p)\leq3$ for all primes $p$.

If $h(G)=2$ then $K=T$ and $G/T\cong\mathbb{Z}^2$.
Hence if $T$ is abelian then $T=1$, by Lemma \ref{metab},
and so $G\cong\mathbb{Z}^2$.

If $h(G)=3$ and the outer action $:G\to{Out(T)}$ determined by
conjugation in $G$ factors through $\mathbb{Z}^2$
then $K\cong{T}\times\mathbb{Z}$.
Thus if $T\not=1$ then 
$\beta_2(G;\mathbb{F}_p)>\beta_1(G;\mathbb{F}_p)$,
for any prime $p$ dividing $|T|$,
by Lemmas \ref{metab1} and \ref{K 2end}.
This contradicts the hypothesis that $G$ has a balanced presentation.

If $h(G)=4$ and $G$ has an abelian normal subgroup $A$ with
such that $G/AT\cong\mathbb{Z}^2$ then $K=AT$,
and $\overline{A}=A/A\cap{T}\cong\mathbb{Z}^2$.
Since $K$ is nilpotent, 
the action of the finite group $T/A\cap{T}$ on $\overline{A}$ is trivial.
Hence $K/A\cap{T}\cong\overline{A}\times(T/A\cap{T})$,
and so $K\cong{T}\rtimes\mathbb{Z}^2$.
Moreover,  if $T\not=1$ then $T\not\leq{K'}$.
Lemmas \ref{metab1} and \ref{K 1end} 
(together with Lemmas \ref{cycad} and \ref{cycad2}) then give 
a similar contradiction.

In parts (2) and (3) the group $G$ is torsion-free,
and so must be one of the known examples given above.
\end{proof}

Imposing a stronger constraint gives a clearer statement.

\begin{cor}
\label{Z2cor}
If $G$ is a homologically balanced nilpotent group with 
an abelian normal subgroup $A$ such that $G/A\cong\mathbb{Z}^2$
then $G\cong\mathbb{Z}^2$,  $\Gamma_q$ (for $q\geq1$)
or $\Omega$. 
\qed
\end{cor}

Note that the second hypotheses in parts (2) and (3) of the theorem 
are not by themselves equivalent to assuming that $G$ is metabelian,
while the hypothesis in the corollary is somewhat stronger.
(On the other hand, it includes all 2-generated metabelian nilpotent groups $G$
with $h(G)>1$.)

The above work leaves open the following questions, for $G$ a nilpotent group with torsion subgroup $T$ and a balanced presentation.
\begin{enumerate}
\item{}if $h(G)=1$ is $H_2(T;\mathbb{Z})=0$?
\item If $h(G)=2, 3$ or 4 and $G$ is metabelian, is $T=1$?
\item{}more ambitiously, if $h(G)>1$ is $T=1$?
\end{enumerate}

\noindent{\bf Acknowledgment.}
I would like to thank Peter Kropholler for reminding me that 
nilpotent groups are often best studied by induction on the abelian case
and Eamonn O'Brien for his advice on $p$-groups.



\begin{thebibliography}{99}

\bibitem{Br} Brown, K. S.  \textit{Cohomology of Groups},

Graduate Texts in Mathematics 87, Springer-Verlag,
Berlin -- Heidelberg -- New York (1982).

\bibitem{HW85} Hausmann, J.-C. and Weinberger, S.  Caract\'eristiques d'Euler et
groupes fondamentaux 
des vari\'et\'es en dimension 4,
Comment. Math. Helv.  60 (1985), 139--144.

\bibitem{HNO'B} Havas, G., Newman, M. F. and O'Brien, E. A. 
Groups of deficiency zero,

in \textit{Geometric and Computational Perspectives on Infinite Groups},
DIMACS Series in Discrete Mathematics and Theoretical Computer Science 25,
Amer. Math. Soc. (1996), 53--67.

\bibitem{Hi87} Hillman, J. A.  The kernel of integral cup product,

J. Austral. Math. Soc. 43 (1987), 10--15.


\bibitem{Hi22} Hillman, J. A. Nilpotent groups with balanced presentations,

J. Group Theory 25 (2022),  713--726.

\bibitem{IZ18} Ivanov, S. O. and Zaikovskii, A.  A. Mod-$2$ (co)homology of an abelian group,

J.  Math. Sciences 252 (2021), 794--803.


\bibitem{Lu83} A. Lubotzky, Group presentation, $p$-adic analytic groups 
and lattices in $SL_2(\mathbb{C})$,

Ann. Math. 118 (1983), 115--130.

\bibitem{Ro} Robinson, D. J. S.  \textit{A Course in the Theory of Groups},

Graduate Texts in Mathematics 80,
Springer-Verlag, Berlin -- Heidelberg -- New York (1982).

\end{thebibliography}
\end{document}